\documentclass{amsart}
\usepackage{amssymb,latexsym, amsmath, amscd}
\usepackage{enumerate}
\usepackage{graphicx}
\usepackage[all]{xy}
\makeatletter
\@namedef{subjclassname@2010}{%
\textup{2010} Mathematics Subject Classification}
\makeatother

\newtheorem{thm}{Theorem}

\newtheorem{lem}[thm]{Lemma}
\newtheorem{prop}[thm]{Proposition}

\theoremstyle{definition}

\numberwithin{equation}{section}

\def\Z{\mathbb Z}

\begin{document}

\title[The pure cactus group is residually nilpotent]{The pure cactus group is residually nilpotent}

\author[J. Mostovoy]{Jacob Mostovoy}

\address{Departamento de Matem\'aticas, CINVESTAV-IPN\\ Col. San Pedro Zacatenco, M\'exico, D.F., C.P.\ 07360\\ Mexico}
\email{jacob@math.cinvestav.mx}

\begin{abstract}
We show that the pure cactus group $\Gamma_{n+1}$ is residually nilpotent and exhibit a surjective homomorphism $\Gamma_{n+1}\to(\Z/2\Z)^{2^n-n(n+1)/2-1}$ whose kernel is residually torsion-free nilpotent. 
\end{abstract}



\maketitle

\section{Introduction: the cactus groups.}
For an integer $n>0$, the \emph{cactus group} $J_n$ has the generators $s_{p,q}$, where $1\leq p< q\leq n$ and the following relations:
\[
\begin{array}{rcll}
s_{p,q}^2&=& 1,&\\
s_{p,q}s_{m,r}&=& s_{m,r}s_{p,q} &\quad \text{if\ } [p,q]\cap [m,r] =\emptyset,\\
s_{p,q}s_{m,r}&=& s_{p+q-r, p+q-m}s_{p,q} &\quad \text{if\ }  [m,r] \subset [p,q].
\end{array}
\]
The group $J_1$ is trivial and $J_2=\Z/2\Z$. Elements of $J_n$ can be drawn as ``planar braids'' with self-intersections in the following manner (the braid below is $s_{3,7}$):

$$\includegraphics[width=80pt]{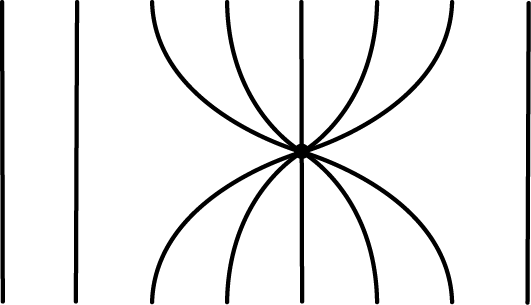}$$

\noindent The product in $J_n$ is simply the concatenation of braids. The relations have the following form:
$$\includegraphics[width=350pt]{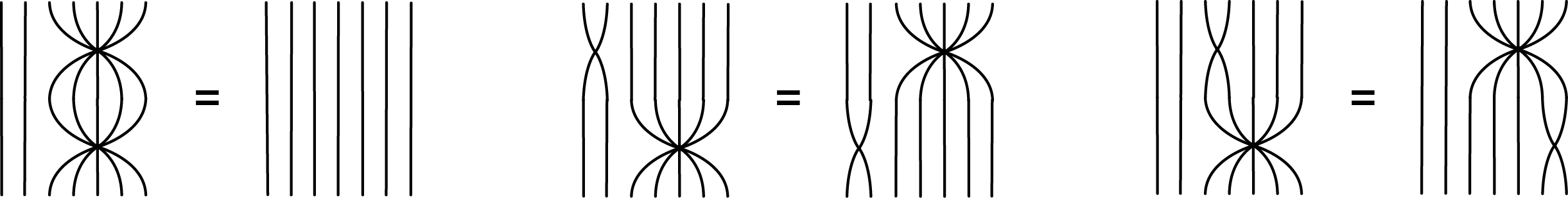}$$

There is a homomorphism $J_n\to S_n$ to the symmetric group: the permutation defined by a braid is obtained by following its strands. The kernel of this homomorphism is the \emph{pure cactus group} $\Gamma_{n+1}$. It is the fundamental group of the real locus of the Deligne-Mumford compactification $\overline{\mathcal{M}}_{0,n+1}$ of the moduli space of rational curves with $n+1$ marked points\footnote{The terminology in the literature varies, see \cite{DJS, De, EHKR, HK}; we use that of \cite{EHKR}.}. 

\medskip

It has been conjectured \cite{EHKR} that $\Gamma_{n+1}$ is residually nilpotent for all $n$. In this note we prove this fact  and show that there is a surjective homomorphism $$\Gamma_{n+1}\to(\Z/2\Z)^{2^n-n(n+1)/2-1}$$  whose kernel is residually torsion-free nilpotent, that is, admits a filtration whose successive quotients are torsion-free, and whose terms intersect in the trivial subgroup.

\section{The proof}

\subsection{The diagram groups.}
The \emph{diagram group} $D_n$ is generated by the $\tau_I$, where $I$ varies over the non-empty subsets of $\{1,\ldots,n\}$, with the relations
\[
\begin{array}{rcll}
\tau_I^2&=& 1,&\\
\tau_I \tau_J &=& \tau_J \tau_I &\quad \text{if\ } I\subset J \text{\ or\ } I\cap J=\emptyset.
\end{array}
\]
A \emph{diagram} is a word in the $\tau_I$; diagrams can be drawn as pictures with $n$ vertical strands and a number of horizontal chords, each at a different height, joining several strands.  For instance, the following diagram has a single chord that joins the strands $1,2$ and $4$:
$$\includegraphics[width=50pt]{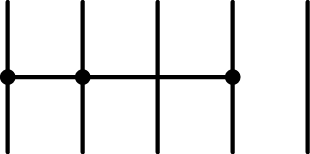}$$

\noindent There is a homomorphism $\Gamma_{n+1}\to D_n$ defined as follows. Write an element of $\Gamma_{n+1}$ as a planar braid on $n$ strands, whose self-intersection points are all on different levels. Consider the vertical lines of a diagram as parametrizing the strands of the braid by height; then, a braid is sent to the diagram whose horizontal chords  join the preimages of the intersection points:

$$\includegraphics[width=180pt]{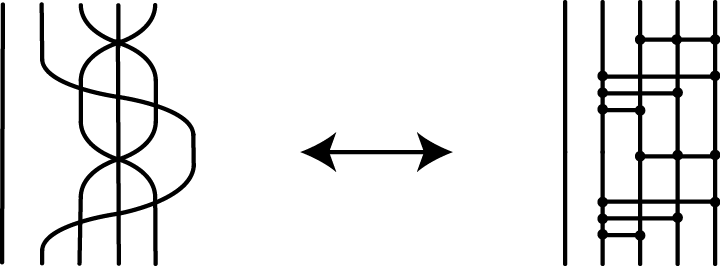}$$

\noindent It is clear that this construction sends the relations in $\Gamma_{n+1}$ to the relations in $D_n$. Moreover, if a word $w$ in the $\tau_I$ is the image of an element in $\Gamma_{n+1}$, the word obtained from $w$ by applying one of the relations of $D_n$  is in the image of the same element. Therefore, 
the pure cactus group $\Gamma_{n+1}$ is a subgroup of $D_n$.

\subsection{The diagram algebra $A_n$}
Let $A_n$ be the unital graded algebra over $\mathbb{F}_2$ generated by the elements $t_I$ of degree one, where $I$ varies over the non-empty subsets of $\{1,\ldots,n\}$, with the relations
\[
\begin{array}{rcll}
t_I^2&=& 0,&\\
t_I t_J &=& t_J t_I &\ \text{if\ } I\subset J \text{\ or\ } I\cap J=\emptyset.
\end{array}
\]
The elements of the form $1+t_I$ are invertible in $A_n$; each is its own inverse. 
\begin{lem}
The subgroup of generated by the $1+t_I$ in the group $A^*_n$ of invertible elements in $A_n$  is isomorphic to the diagram group $D_n$.
\end{lem}
\begin{proof}
Call a word in the generators $\tau_I$ \emph{lean} if, by applying solely the commutation relations 
$\tau_I\tau_J=\tau_J\tau_I$ it cannot be taken into the form which contains some generator repeated twice consecutively. It can be shown that any two lean words which represent the same diagram in $D_n$ can be transformed into each other by means of the commutation relations only. 

Each element of $D_n$ can be represented by a lean word and a non-trivial lean word represents a non-trivial diagram. Quite similarly, one can speak about lean monomials (that is, products of generators) in $A_n$; lean words in the $\tau_I$ are in one-to-one correspondence with lean monomials in the $t_I$ and each non-trivial lean monomial is non-zero in $A_n$.

Consider the homomorphism $D_n\to A^*_n$ which sends $\tau_I$ to $1+t_I$. The image of a lean word in the generators $\tau_I$ is equal, up to terms of lower degree, to the corresponding lean monomial in the $t_I$ and, hence, is non-trivial.
\end{proof}

The group $A^*_n$ of invertible elements in $A_n$ is residually nilpotent. Indeed, let $(A^*_n)_k$ be the subgroup of $A^*_n$ consisting of the elements of the form $$1+\text{terms\ of\ degree\ at\ least\ }k.$$ The inverse of an element whose homogeneous term of the lowest non-zero degree equals $u$ also is of the form $1+u+\text{terms\ of\ higher\ degree}$. As a consequence, the commutator of 
$(A^*_n)_k$ and $A^*_n$ lies in $(A^*_n)_{k+1}$; the intersection of all the $(A^*_n)_k$ is trivial. Since $\Gamma_{n+1}\subset D_n\subseteq A^*_n$, we get
\begin{thm}
The pure cactus group $\Gamma_{n+1}$ is residually nilpotent.
\end{thm}
\subsection{The even diagram subgroups}
For each non-empty $I\subseteq \{1,\ldots,n\}$ define the homomorphism
$$\delta_I: D_n\to \Z/2\Z$$
by $\delta_I(\tau_I)=1$ and $\delta_I(\tau_J)=0$ if $J\neq I$. 
Call the intersection of the kernels of all the $\delta_I$ the \emph{even diagram subgroup} of $D_n$ and denote it by $D_n^{\circ}$. 
Define $$\Gamma_{n+1}^{\circ}=\Gamma_{n+1}\cap D_n^{\circ}.$$
\begin{prop}
The subgroup $\Gamma_{n+1}^{\circ}$ is the kernel of the surjective homomorphism
$$\prod_{|I|>2}\delta_I: \Gamma_{n+1}\to (\Z/2\Z)^{2^n-n(n+1)/2-1}.$$
\end{prop}
\begin{proof}
Consider a subset $I$ with $k>2$ elements. Let $\sigma\in S_n$ be a permutation which maps the elements of $I$ to the set $\{1,\ldots, k\}$. Write $\lambda$ for the permutation which writes $1,2, \ldots, k$ backwards. The subgroup of $J_n$ generated by the elements of the form $s_{i,i+1}$ maps surjectively onto the symmetric group $S_n$ since the latter is generated by transpositions. Choose in this subgroup an element $\overline{\sigma}$ which maps to  $\sigma$ and $\overline{\sigma^{-1}\lambda}$ which maps to $\sigma^{-1}\lambda$. Then, $$s:=\overline{\sigma^{-1}\lambda}\cdot s_{1,k}\cdot \overline{\sigma}\,\in\,\Gamma_{n+1}$$ and the diagram of $s$ is a product of a number of chords, connecting two strands each, and precisely one chord that connects the strands indexed by the subset $I$. 
In particular, $\delta_I(s) = 1$ and $\delta_J(s) = 0$ if $J\neq I$ and $|J|>2$. This shows that the product of all the $\delta_I$ with $I$ of cardinality greater than two is surjective.
There are exactly $2^n-n(n+1)/2-1$ subsets $I$ of this kind.

Now, if $b\in\Gamma_{n+1}$ is thought of as a braid, any pair of its strands intersect an even number of times. If $b$ is in 
the kernel of $\prod_{|I|>2}\delta_I$, each pair of strands of $b$ intersect at an even number of points where more than two strands meet. As a consequence, each $\tau_{I}$ with $|I|=2$ appears in the diagram of $b$ an even number of times and, therefore, $\delta_I(b)=0$. This shows that $\Gamma_{n+1}^{\circ}$ coincides with the kernel of $\prod_{|I|>2}\delta_I$. 
\end{proof}

\begin{thm}
The group $\Gamma_{n+1}^{\circ}$ is residually torsion-free nilpotent.
\end{thm}
\begin{proof}
It is sufficient to show that the even diagram subgroup $D_n^{\circ}$ is residually torsion-free nilpotent. Let $A_n'$ be the algebra of formal power series with integer coefficients in the non-commuting variables $t_I$, where $I$, as before, varies over the non-empty subsets of  $\{1,\ldots, n\}$, and the $t_I$ satisfy the relation
$$t_I t_J = t_J t_I \quad \text{if\ } I\subset J \text{\ or\ } I\cap J=\emptyset.$$
Define a homomorphism of
$D_n^\circ$ to the group $(A_n')^*$ of invertible elements of $A_n'$ by sending each odd-numbered occurrence of $\tau_I$ to $1+t_I$ and each even-numbered occurrence of $\tau_I$ to $(1+t_I)^{-1}$. This map is injective: the image of a lean word in the $\tau_I$ contains the corresponding monomial in the $t_I$ with the coefficient $(-1)^{d/2}$, where $d$ is the number of letters in the word (in other words, the number of chords in the corresponding diagram). Since the filtration of $(A_n')^*$ by the lowest degree of the non-trivial terms has torsion-free quotients and trivial intersection, we see that $D_n^\circ$ also carries such a filtration and, therefore, is residually torsion-free nilpotent.
\end{proof}


\begin{thebibliography}{999}

\bibitem{DJS} M. Davis, T. Januszkiewicz, R. Scott, Fundamental groups of blow-ups, Advances in Mathematics,
{\bf 177}  (2003), 115-179.

\bibitem{De} S.~Devadoss, Tessellations of Moduli Spaces and the Mosaic Operad, Contemporary Mathematics, {\bf 239} (1999), 91-114.

\bibitem{EHKR} P.~Etingof, A.~Henriques, J.~Kamnitzer, E.~Rains,  The cohomology ring of the real locus of the moduli space of stable curves of genus 0 with marked points, 
Annals of Mathematics, {\bf 171} (2010), 731--777.

\bibitem{HK} A.~Henriques, J.~Kamnitzer, Crystals and coboundary categories, Duke Mathematical Journal,
{\bf 132} (2006), 191-216. 
\end{thebibliography}
\end{document}